\documentclass[12pt]{scrartcl}

\usepackage[]{amsmath, amssymb,amsfonts,amsthm,mathtools,braket,color,enumerate}

\usepackage[margin=20mm]{geometry}

\usepackage{hhline}
\usepackage{url}
\usepackage{here}
\usepackage{tikz-cd}
\usepackage{tikz}
\usetikzlibrary{decorations}
\usetikzlibrary{arrows}
\tikzstyle{v} = [circle, draw, inner sep=2pt, minimum size=3pt, fill=black]
\tikzstyle{l} = [rectangle, draw, rounded corners]

\usepackage{tabularx}
\newcolumntype{L}[1]{>{\raggedright\arraybackslash}p{#1}}
\newcolumntype{C}[1]{>{\centering\arraybackslash}p{#1}}
\newcolumntype{R}[1]{>{\raggedleft\arraybackslash}p{#1}}

\theoremstyle{plain}% default
\newtheorem{theorem}{Theorem}[section]
\newtheorem{lemma}[theorem]{Lemma}
\newtheorem{proposition}[theorem]{Proposition}
\newtheorem{corollary}[theorem]{Corollary}

\theoremstyle{definition}
\newtheorem{definition}[theorem]{Definition}

\newtheorem{example}[theorem]{Example}

\newtheorem{remark}[theorem]{Remark}

\newcommand{\FinSetE}{\mathbb{E}}
\newcommand{\FinSetB}{\mathbb{B}}
\newcommand{\FinLinSetL}{\mathbb{L}}

\newcommand{\SpLShi}[2]{L_{#1}\mathcal{S}^{#2}}
\newcommand{\SpLCat}[2]{L_{#1}\mathcal{C}^{#2}}
\newcommand{\SpLK}[2]{L_{#1}K^{#2}}

\newcommand{\SpSet}{\mathsf{E}}
\newcommand{\SpList}{\mathsf{L}}
\newcommand{\SpCycle}{\mathsf{C}}
\newcommand{\SpPart}{\Pi}
\newcommand{\SpF}{\mathsf{F}}
\newcommand{\SpG}{\mathsf{G}}

\DeclareMathOperator{\Map}{Map}
\DeclareMathOperator{\gap}{gap}
\DeclareMathOperator{\level}{level}

\newcommand{\ic}{\textcolor{white}{,}}

\title {Enumeration of Flats of the Extended Catalan and Shi Arrangements with Species}
\author{Norihiro Nakashima\thanks{Department of Mathematics, Nagoya Institute of Technology, Aichi, 466-8555, Japan. Email: nakashima@nitech.ac.jp} \and
Shuhei Tsujie
\thanks{Department of Mathematics, Hokkaido University of Education, Asahikawa, Hokkaido 070-8621, Japan. 
Email:tsujie.shuhei@a.hokkyodai.ac.jp}
}

\date{}
\begin{document}
\maketitle
	
\begin{abstract}
The number of flats of a hyperplane arrangement is considered as a generalization of the Bell number and the Stirling number of the second kind. 
Robert Gill gave the exponential generating function of the number of flats of the extended Catalan arrangements, using species. 
In this article, we introduce the species of flats of the extended Catalan and Shi arrangements and they are given by iterated substitution of species of sets and lists. 
Moreover, we enumerate the flats of these arrangements in terms of infinite matrices. 
\end{abstract}


{\footnotesize \textit{2020 MSC}: 
05A18, %Pratitions of sets [05Axx Enumerative combinatorics]
05A19, %Combinatorial identities, bijective combinatorics [05Axx Enumerative combinatorics]
05C22, %Signed and weighted graphs [05Cxx Graph theory]
52C35 %Arrangements of points, flats, hyperplanes [52Cxx Discrete geometry]
%32S22,  %Relations with arrangements of hyperplanes [32Sxx Singularities]
%05C15, %Coloring of graphs and hypergraphs [05Cxx Graph theory]
%20F55,  %Reflection and Weyl groups [20Fxx Special aspects of infinite or finite groups]
%13N15 %Derivations [13Nxx Differential algebra]
}

\tableofcontents

\section{Introduction}
A \textbf{hyperplane arrangement} is a finite collection of affine subspaces of codimension $ 1 $ in an affine space over an arbitrary field $ \mathbb{K} $. 
In spite of its simple definition, arrangements are investigated in a variety of ways, such as topological, algebrogeometric, and combinatorial aspects. 
A standard reference for hyperplane arrangements is the text written by Orlik and Terao \cite{orlik1992arrangements}. 

Given an arrangement $ \mathcal{A} $, let $ L(\mathcal{A}) $ denote the set of nonempty intersections of hyperplanes in $ \mathcal{A} $. 
Note that the ambient space is a member of $ L(\mathcal{A}) $ since it is regarded as the intersection over the empty set. 
We call an element of $ L(\mathcal{A}) $ a \textbf{flat}. 
Define a partial order on $ L(\mathcal{A}) $ by the reverse inclusion, that is, $ X \leq Y \Leftrightarrow X \supseteq Y $ for $X,Y\in L(\mathcal{A})$.  We call $ L(\mathcal{A}) $ the \textbf{intersection poset} of $ \mathcal{A} $.  
This poset plays an important role in the theory of hyperplane arrangements. 
For each nonnegative integer $ k $, let 
\begin{align*}
L_{k}(\mathcal{A}) \coloneqq \Set{ X \in L(\mathcal{A}) | \dim X = k }. 
\end{align*}
When $ \mathcal{A} $ is \textbf{central}, that is, the intersection of all hyperplanes in $ \mathcal{A} $ is nonempty, the poset $ L(\mathcal{A}) $ is a geometric lattice. 

A \textbf{set partition} of a finite set $ V $ is a collection $ \pi = \{B_{1}, \dots, B_{k}\} $ of nonempty subsets $ B_{i} \subseteq V $ such that $ B_{i} \cap B_{j} = \varnothing $ for $ i \neq j $ and $ \bigcup_{i=1}^{k}B_{i}=V $. 
Each $ B_{i} $ is called a \textbf{block} of $ \pi $. 
Let $ \pi $ and $ \pi^{\prime} $ be set partitions of $ V $.
Define a partial order $ \pi \leq \pi^{\prime} $ if
each block of $ \pi $ is a subset of some block of $ \pi^{\prime} $.
We also say that $ \pi $ \textbf{refines} $ \pi^{\prime} $ if $ \pi \leq \pi^{\prime} $. Then the collection of the set partitions of $ V $ forms a lattice.

For a positive integer $ n $, let $ [n] $ be the set $ \{1, \dots, n\} $
and $[0]$ the empty set.
The number of set partitions of $ [n] $ is called the \textbf{Bell number}, denoted by $ B(n) $, where $ B(0)=1 $. The number of set partitions of $ [n] $ into $ k $ blocks is called the \textbf{Stirling number of the second kind}, denoted by $ S(n, k) $, where $ S(0,0)=1 $. 

Let $ x_{1}, \dots, x_{n} $ denote coordinates of $ \mathbb{R}^{n} $ and $ \mathcal{B}_{n} $ the $ n $-dimensional \textbf{braid arrangement} (also known as the \textbf{Weyl arrangement of type $ A_{n-1} $}), which consists of hyperplanes $ \{x_{i}-x_{j}=0\} $ in $\mathbb{R}^n$ for $ 1 \leq i < j \leq n $. 
It is well known that there exists an isomorphism from $ L(\mathcal{B}_{n}) $ to the lattice of set partitions of $ [n] $ which sends a $ k $-dimensional flat to a set partition into $ k $ blocks. 
In other words, $ |L(\mathcal{B}_{n})| = B(n) $ and $ |L_{k}(\mathcal{B}_{n})| = S(n,k) $. 
The numbers of flats and $ k $-dimensional flats of an arrangement is considered to be generalizations of the Bell numbers and the Stirling numbers. 

Define the \textbf{extended Catalan arrangement} $ \mathcal{C}_{n}^{m} $ and the \textbf{extended Shi arrangement} $ \mathcal{S}_{n}^{m} $ in $ \mathbb{R}^{n} $ as follows. 
\begin{align*}
\mathcal{C}_{n}^{m} &\coloneqq \Set{\{x_{i}-x_{j}=a\} |1\leq i<j\leq n, -m \leq a \leq m}, \quad (m \geq 0), \\
\mathcal{S}_{n}^{m} &\coloneqq \Set{\{x_{i}-x_{j}=a\} |1\leq i<j\leq n, 1-m \leq a \leq m}, \quad (m \geq 1). 
\end{align*}

Note that $ \mathcal{C}_{n}^{0} = \mathcal{B}_{n} $ for every nonnegative integer $ n $ and $ L_{0}(\mathcal{C}^{m}_{n}) = L_{0}(\mathcal{S}^{m}_{n}) = \varnothing $ unless $ n = 0 $. 
Gill \cite{gill1998number-dm} investigated the intersection posets of the extended Catalan arrangements. 
First Gill determined the number of maximal elements of the poset $ L(\mathcal{C}^{m}_{n}) $ as follows. 
\begin{theorem}[Gill {\cite[Theorem 1]{gill1998number-dm}}]\label{Gill1}
Let $ m $ be a nonnegative integer. 
Then 
\begin{align*}
 \sum_{n=1}^{\infty}\left|L_{1}(\mathcal{C}_{n}^{m})\right|\frac{x^{n}}{n!} = \frac{e^{x}-1}{1-m(e^{x}-1)}. 
\end{align*}
\end{theorem}

In Gill's works \cite{gill1998number-dm}, the use of species, which were initiated by Joyal \cite{joyal1981theorie-aim}, is a noteworthy point. 
Standard references for species are the texts written by Bergeron, Labelle, and Leroux \cite{bergeron1998combinatorial,bergeron2013introduction}. 
Let $ \FinSetB $ denote the category of finite sets and bijections and $ \FinSetE $ the category of finite sets and maps. 
A \textbf{species, or $ \FinSetB $-species,} is a functor $ \SpF \colon \FinSetB \rightarrow \FinSetE $. 
The value of a species $ \SpF $ at a finite set $ V $ is denoted by $ \SpF[V] $. 
Moreover, we write simply $ \SpF[n] $ instead of $ \SpF[[n]] $. 
The symbol $ \SpF(x) $ is used for the exponential generating function 
\begin{align*}
\SpF(x) \coloneqq \sum_{n=0}^{\infty}\left|\SpF[n]\right|\frac{x^{n}}{n!}.
\end{align*}

\begin{example}
Let $ \SpSet $ denote the \textbf{species of sets}. 
Namely $ \SpSet[V] \coloneqq \{ V \} $. 
Then we have $ |\SpSet[n]|=1 $ for every $ n $ and thus $ \SpSet(x) = \sum_{n=0}^{\infty}\frac{x^{n}}{n!} = e^{x} $. 
\end{example}

\begin{example}
Let $ \SpList $ denote the \textbf{species of lists}. 
Namely 
\begin{align*}
\SpList[V] \coloneqq \Set{(v_{1}, \dots, v_{n}) | n = |V| \text{ and } V = \{v_{1}, \dots, v_{n}\}}. 
\end{align*}
We have $ \left|\SpList[n]\right|=n! $ for every $ n $ and hence $ \SpList(x) = \sum_{n=0}^{\infty}x^{n} = \frac{1}{1-x} $. 
\end{example}

Making use of species enables us to calculate generating functions systematically. 
Namely we can define operations for species such as sum, product, and so on.
These operations are compatible with the corresponding operations for the generating functions. 
Hence we may say that species is a refinement of the exponential generating function. 
\begin{definition}
For species $ \SpF $ and $ \SpG $, we define the \textbf{sum} $ \SpF+\SpG $ by $ (\SpF+\SpG)[V] \coloneqq \SpF[V] \sqcup \SpG[V] $, where $ \sqcup $ means the disjoint union. 
\end{definition}
Then we have $ (\SpF+\SpG)(x) = \SpF(x)+\SpG(x) $. 
\begin{definition}
For every species $ \SpF $ and a nonnegative integer $ k $, define the species $ \SpF_{k} $ by 
\begin{align*}
\SpF_{k}[V] \coloneqq \begin{cases}
\SpF[V], & \text{ if } |V|=k; \\
\varnothing, & \text{ otherwise. }
\end{cases}
\end{align*}
\end{definition}
Every species $ \SpF $ has a canonical decomposition $ \SpF = \sum_{k=0}^{\infty}\SpF_{k} $.
Furthermore, if we write $ \SpF_{+} \coloneqq \sum_{k=1}^{\infty}\SpF_{k} $, then $ \SpF_{+}[\varnothing] = \varnothing $ and $ \SpF_{+}(x) = \SpF(x)-\SpF(0) $.
\begin{example}\label{ex:sum}
\begin{align*}
\SpSet_{+}(x) = e^{x}-1, \quad \SpSet_{k}(x) = \frac{x^{k}}{k!}, \quad \text{ and } \quad \SpList_{+}(x) = \frac{1}{1-x}-1 = \frac{x}{1-x}. 
\end{align*}
\end{example}
In this article, we frequently use the operation $ \SpF \circ \SpG $ of species, called \textbf{substitution} (or \textbf{composition}).
\begin{definition}
Let $ \SpF $ and $ \SpG $ be species with $ \SpG[\varnothing] = \varnothing $. 
Define
\begin{align*}
(\SpF \circ \SpG)[V] \coloneqq \bigsqcup_{\pi \in \SpPart[V]} \left(\SpF[\pi] \times \prod_{B \in \pi}\SpG[B]\right), 
\end{align*}
where $ \SpPart $ denotes the \textbf{species of set partitions}. 
\end{definition}
Substitution of species corresponds to substitution of generating functions, that is, $ (\SpF\circ \SpG)(x) = \SpF(\SpG(x)) $. 
It is a functional tool for computing the generating function and explains why the exponential function $ e^{x} $ sometimes appears in the generating function. 
\begin{example}\label{ex:composition sets}
By definition, the exponential generating function of the Bell numbers is given by $ \SpPart(x) $. 
The species of set partitions $ \SpPart $ coincides with the species $ \SpSet \circ \SpSet_{+} $ (See Example \ref{ex:substitution} for details). 
Hence we have the following. 
\begin{align*}
\SpPart(x) = (\SpSet \circ \SpSet_{+})(x) = \exp\left( e^{x}-1 \right). 
\end{align*}
\end{example}
\begin{example}\label{ex:composition lists}
Given a species $\SpF$, let $ \SpF_{+}^{\circ m} $ be the $ m $-times iterated self-substitution of $ \SpF_{+} $. 
For the species $ \SpList $ of lists, we have 
\begin{align*}
\SpList_{+}^{\circ 2}(x) = (\SpList_{+} \circ \SpList_{+})(x) = \frac{\frac{x}{1-x}}{1-\frac{x}{1-x}} = \frac{x}{1-2x}. 
\end{align*}
More generally, one can show easily that 
\begin{align*}
\SpList_{+}^{\circ m}(x) = \frac{x}{1-mx}. 
\end{align*}
\end{example}

The construction of the extended Catalan arrangement $ \mathcal{C}^{m}_{n} $ is functorial in $ n $.
Namely, for every finite set $ V $, we may construct the corresponding Catalan arrangement in the vector space $ \mathbb{R}^{V} = \Map(V,\mathbb{R}) $ and hence there exist species $ \SpLCat{}{m} $ and $ \SpLCat{k}{m} $ such that $ \SpLCat{}{m}[n] = L(\mathcal{C}^{m}_{n}) $ and $ \SpLCat{k}{m}[n] = L_{k}(\mathcal{C}^{m}_{n}) $. 
Gill's second theorem is as follows. 
\begin{theorem}[Gill {\cite[Theorem 2]{gill1998number-dm}}]\label{Gill2}
For every $ m \geq 0 $, the equality $ \SpLCat{}{m} = \SpSet \circ \SpLCat{1}{m} $ holds. 
Moreover, the bivariate generating function of $ \left|L_{k}(\mathcal{C}_{n}^{m})\right| $ is given by
\begin{align*}
 \sum_{n=0}^{\infty}\sum_{k=0}^{n} \left|L_{k}(\mathcal{C}_{n}^{m})\right|t^{k}\frac{x^{n}}{n!} = \exp\left(t \frac{e^{x}-1}{1-m(e^{x}-1)}\right). 
\end{align*}
\end{theorem}

We write $ \SpF = \SpG $ if there exists a natural isomorphism between species $ \SpF $ and $ \SpG $. 
In this paper, we will improve Gill's results in terms of species as follows. 
\begin{theorem}\label{main catalan}
Let $ m $ and $ k $ be nonnegative integers. 
Then 
\begin{align*}
\SpLCat{k}{m} = \SpSet_{k} \circ \SpList_{+}^{\circ m} \circ \SpSet_{+} \quad 
\text{ and } \quad \SpLCat{}{m} = \SpSet\circ \SpList_{+}^{\circ m} \circ \SpSet_{+}. 
\end{align*}
\end{theorem}
This theorem together with Example \ref{ex:sum}, Example \ref{ex:composition sets}, and Example \ref{ex:composition lists} leads to the following corollary, which is equivalent to the generating functions in Theorem \ref{Gill1} and Theorem \ref{Gill2}. 
\begin{corollary}
\begin{align*}
\SpLCat{k}{m}(x) &= \sum_{n=0}^{\infty} \left| L_{k}(\mathcal{C}^{m}_{n}) \right|\frac{x^{n}}{n!} 
= \frac{1}{k!}\left(\frac{e^{x}-1}{1-m(e^{x}-1)}\right)^{k}, \\
\SpLCat{}{m}(x) &= \sum_{n=0}^{\infty} \left| L(\mathcal{C}^{m}_{n}) \right|\frac{x^{n}}{n!} 
= \exp\left(\frac{e^{x}-1}{1-m(e^{x}-1)}\right). 
\end{align*}
\end{corollary}

We will also give an analog for the extended Shi arrangement $ \mathcal{S}^{m}_{n} $. 
However, the construction of $ \mathcal{S}^{m}_{n} $ cannot be regarded as a functor on $ \FinSetB $ since the construction requires the linear order on the set $ [n] $. 
For this reason we consider $ \FinLinSetL $-species, that is, a functor $ \SpF \colon \FinLinSetL \rightarrow \FinSetE $, where $ \FinLinSetL $ denotes the category of linearly ordered finite sets and order-preserving bijections. 
Note that every $ \FinSetB $-species can be considered as an $ \FinLinSetL $-species via the forgetful functor from $ \FinLinSetL $ to $ \FinSetB $. 

Let $ \SpLShi{}{m} $ and $\SpLShi{k}{m} $ denote the $ \FinLinSetL $-species such that $ \SpLShi{}{m}[n] = L(\mathcal{S}^{m}_{n}) $ and $ \SpLShi{k}{m}[n] = L_{k}(\mathcal{S}^{m}_{n}) $.
The following is another main result of this article. 
\begin{theorem}\label{main shi}
Let $ m $ and $ k $ be nonnegative integers. 
Then 
\begin{align*}
\SpLShi{k}{m} = \SpSet_{k} \circ \SpList_{+}^{\circ m} \quad 
\text{ and } \quad \SpLShi{}{m} = \SpSet\circ \SpList_{+}^{\circ m}. 
\end{align*}
\end{theorem} 
\begin{corollary}
\begin{align*}
\SpLShi{k}{m}(x) &= \sum_{n=0}^{\infty} \left| L_{k}(\mathcal{S}^{m}_{n}) \right|\frac{x^{n}}{n!} 
= \frac{1}{k!}\left(\frac{x}{1-mx}\right)^{k}, \\
\SpLShi{}{m}(x) &= \sum_{n=0}^{\infty} \left| L(\mathcal{S}^{m}_{n}) \right|\frac{x^{n}}{n!} 
= \exp\left(\frac{x}{1-mx}\right). 
\end{align*}
\end{corollary}

Moreover, we will give explicit formulas for the numbers $ |L_{k}(\mathcal{S}^{m}_{n})| $ and $ |L_{k}(\mathcal{C}^{m}_{n})| $ of $ k $-dimensional flats of $ n $-dimensional extended Catalan and extended Shi arrangements with infinite matrices. 
Let $ [a_{ij}] $ denote the infinite matrix whose entry in the $ i $-th row and the $ j $-th column is $ a_{ij} $, where $ i $ and $ j $ run over all positive integers. 
Let 
\begin{align*}
c \coloneqq \left[ c(j,i) \right] \quad \text{ and } \quad S \coloneqq \left[ S(j,i) \right], 
\end{align*}
where $ c(j,i) $ denote the \textbf{unsigned Stirling number of the first kind}, that is, the number of ways to partition a $ j $-element set into $ i $ cycles. 
Note that most of tables, including Table \ref{tab:triangles}, consisting of such numbers are lower triangular. 
However, our infinite matrices $ c $ and $ S $ are transposed and hence upper triangular. 
Namely, 
\begin{align*}
c = \begin{bmatrix}
1 & 1 & 2 & 6 & 24 & \cdots \\
0 & 1 & 3 & 11 & 50 & \cdots \\
0 & 0 & 1 & 6 & 35 & \cdots \\
0 & 0 & 0 & 1 & 10 & \cdots \\
0 & 0 & 0 & 0 & 1 & \cdots \\
\vdots & \vdots & \vdots & \vdots & \vdots & \ddots
\end{bmatrix}, \qquad
S = \begin{bmatrix}
1 & 1 & 1 & 1 & 1 & \cdots \\
0 & 1 & 3 & 7 & 15 & \cdots \\
0 & 0 & 1 & 6 & 25 & \cdots \\
0 & 0 & 0 & 1 & 10 & \cdots \\
0 & 0 & 0 & 0 & 1 & \cdots \\
\vdots & \vdots & \vdots & \vdots & \vdots & \ddots
\end{bmatrix}. 
\end{align*}
 
\begin{theorem}\label{main matrices}
\begin{align*}
\Big[ \, \big|L_{i}(\mathcal{C}^{m}_{j})\big| \, \Big] = (Sc)^{m}S \quad \text{ and } \quad \Big[ \, \big|L_{i}(\mathcal{S}^{m}_{j})\big| \, \Big] = (Sc)^{m}. 
\end{align*}
\end{theorem}
From Theorem \ref{main matrices}, we may also calculate the matrices recursively as follows. 
\begin{align*}
\begin{array}{ccccccccccc}
\mathcal{C}^{0} & \stackrel{c}{\longrightarrow} & \mathcal{S}^{1} & \stackrel{S}{\longrightarrow} & \mathcal{C}^{1} & \stackrel{c}{\longrightarrow} & \mathcal{S}^{2} & \stackrel{S}{\longrightarrow} & \mathcal{C}^{2} & \stackrel{c}{\longrightarrow} & \cdots \\
\hspace{-2mm}S && \hspace{-2mm}Sc && \hspace{-2mm}ScS && \hspace{-2mm}ScSc && \hspace{-2mm}ScScS &&
\end{array}
\end{align*}

The \textbf{Lah number} is the number of ways to partition an $ n $-element set into $ k $ nonempty lists, which is equal to the cardinality of $ (\SpSet_{k} \circ \SpList_{+})[n] $ (See Example \ref{ex:substitution} for details). 
It is well known that the Lah number is given by the following formula. 
\begin{proposition}[See {\cite[p.44]{riordan2002introduction}} for example]\label{lah number}
Let $ k $ and $ n $ be nonnegative integers. 
Then
\begin{align*}
\big| (\SpSet_{k} \circ \SpList_{+})[n] \big| = \dfrac{n! (n-1)!}{k!(k-1)!(n-k)!}. 
\end{align*}
\end{proposition}

Using the Lah numbers, we give an explicit formula for the number of flats of the extended Shi arrangement. 
\begin{theorem}\label{main shilah}
\begin{align*}
\big|L_{k}(\mathcal{S}^{m}_{n})\big| = m^{n-k}\dfrac{n! (n-1)!}{k!(k-1)!(n-k)!}. 
\end{align*}
\end{theorem}

The organization of this paper is as follows. 

In \S \ref{sec:species}, we give some examples of substitutions of species. 
Also, we introduce tree notation for the species $ \SpList_{+}^{\circ m} $. 
Although actually this notation is not required for the proofs of our main theorems, it helps us to recognize elements of $ \SpList_{+}^{\circ m} $. 

In \S \ref{sec:gain graph}, we review the theory of gain graphs developed by Zaslavsky. 
Since the extended Catalan and Shi arrangements are expressed by using gain graphs, the intersection posets of these arrangements can be represented by a kind of partitions of the vertices of the corresponding gain graphs. 
This guarantees that it suffices to know the $ 1 $-dimensional flats in order to know all flats, that is, $ \SpLCat{}{m} = \SpSet \circ \SpLCat{1}{m} $ and $ \SpLShi{}{m} = \SpSet \circ \SpLShi{1}{m} $. 

Finally in \S \ref{sec:proofs}, we will give proofs of Theorem \ref{main catalan}, Theorem \ref{main shi}, Theorem \ref{main matrices}, and Theorem \ref{main shilah}. 

%In \S \ref{sec:bell polynomial}, we state a relation between species of the form $ \SpSet_{k} \circ \SpF  $ and the partial Bell polynomials. 

In \S \ref{sec:table}, as an appendix we give numerical tables for the number of flats with ID numbers in the On-Line Encyclopedia of Integer Sequences (OEIS) \cite{OEIS}. 

Note that, throughout of this paper, we will construct several natural transformations. 
However, we will omit proofs of commutativity of diagrams for the natural transformations since they are obvious. 

\section{Substitution of species and tree notation}\label{sec:species}
\subsection{Examples}

Let $ \SpF $ and $ \SpG $ be species with $ \SpG[\varnothing] = \varnothing $. 
Recall the definition of the substitution $ \SpF \circ \SpG $. 
\begin{align*}
(\SpF \circ \SpG)[V] \coloneqq \bigsqcup_{\pi \in \SpPart[V]} \left(\SpF[\pi] \times \prod_{B \in \pi}\SpG[B]\right), 
\end{align*}
where $ \SpPart $ denotes the species of set partitions. 

The substitution $ \SpF \circ \SpG $ has the external and internal structure. 
Namely, for each set partition $ \pi $, we take an element of $ \SpF[\pi] $, which is the external structure. 
For every block $ B $ of $ \pi $, we choose an element of $ \SpG[B] $, which is the internal structure. 
In the usual case, a constituent of the external structure is labeled with blocks of a partition and if we ``substitute" the labels with internal structures, then we obtain an element of $ (\SpF \circ \SpG)[V] $. 

\begin{example}\label{ex:substitution}
Let $ \SpG $ be a species such that if $ g \in \SpG[B] , g^{\prime} \in \SpG[B^{\prime}]  $, and $ g = g^{\prime} $, then $ B=B^{\prime} $. 
We will see that the species $ \SpSet \circ \SpG $ may be considered as a species of set partitions consisting of $ \SpG $-structures. 
By definition,
\begin{align*}
(\SpSet \circ \SpG)[V] &= \bigsqcup_{\pi \in \SpPart[V]} \left(\SpSet[\pi] \times \prod_{B \in \pi}\SpG[B]\right) = \bigsqcup_{\pi \in \SpPart[V]} \left(\{ \pi \}\times \prod_{B \in \pi}\SpG[B]\right). 
\end{align*}
For any set partition $ \pi = \{B_{1}, \dots, B_{k}\} $, every element $ \left(\pi , (g_{B_{1}}, \dots, g_{B_{k}}) \right) \in \{ \pi \}\times \prod_{B \in \pi}\SpG[B] $ can be identified with the set $ \{g_{B_{1}}, \dots, g_{B_{k}}\} $. 
Then the set $ (\SpSet \circ \SpG)[V] $ is identified with 
\begin{align*}
\Set{ \{g_{B_{1}}, \dots, g_{B_{k}}\} | \{B_{1}, \dots, B_{k}\} \in \SpPart[V], \ g_{B_{i}} \in \SpG[B_{i}]  }. 
\end{align*}
In particular, $ \SpSet \circ \SpSet_{+} $ gives set partitions consisting of sets, that is, ordinary set partitions. 
Namely $ \SpSet \circ \SpSet_{+} = \SpPart $. 
Therefore $ \big| (\SpSet_{k} \circ \SpSet_{+})[n] \big| $ is the Stirling number of the second kind. 
For the same reason $ \big| (\SpSet_{k} \circ \SpList_{+})[n] \big| $ yields the Lah number. 
\end{example}

We sometimes omit commas in sets, lists, and so on.
For example, we write $ \{123\} $ instead of $ \{1,2,3\} $.

\begin{example}
Let $ \SpG $ be as above. 
By the similar discussion, we may say that the species $ \SpList \circ \SpG $ may be considered as a species of lists consisting of $ \SpG $-structures. 
Namely the set $ (\SpList \circ \SpG)[V] $ can be identified naturally with 
\begin{align*}
\Set{ \left(g_{B_{1}}, \cdots, g_{B_{k}}\right) | \{B_{1}, \dots, B_{k}\} \in \SpPart[V], \ g_{B_{i}} \in \SpG[B_{i}] }. 
\end{align*}
The species $ \SpList \circ \SpSet_{+} $ is known as a species of \textbf{set compositions} (or \textbf{ordered set partitions}). 
The cardinality $ \left|(\SpList \circ \SpSet_{+})[n]\right| $ is called the \textbf{ordered Bell number} (or \textbf{Fubini number}). 
For instance $ \left|(\SpList \circ \SpSet_{+})[3]\right| = 13 $ since it consists of the following set compositions. 
\begin{align*}
&(\{123\}), (\{12\} \{3\}), (\{13\} \{2\}), (\{23\} \{1\}), (\{3\} \{12\}), (\{2\} \{13\}), (\{1\} \{23\}),  \\
&(\{1\}\{2\}\{3\}),(\{1\}\{3\}\{2\}),(\{2\}\{1\}\{3\}),(\{2\}\{3\}\{1\}),(\{3\}\{1\}\{2\}),(\{3\}\{2\}\{1\}). 
\end{align*}
\end{example}

\begin{example}
The species $ \SpSet_{1} $ of singletons behaves as the identity element with respect to substitution. 
Namely, $ \SpSet_{1} \circ \SpG = \SpG $ and $ \SpF \circ \SpSet_{1} = \SpF $. 
\end{example}

A lot of researchers have been studied iterated substitutions of species of sets and lists. 
For example, Motzkin \cite{motzkin1971sorting-c} investigated several structures including, ``sets of sets" $ \SpSet \circ \SpSet_{+} $, ``sets of lists" $ \SpSet \circ \SpList_{+} $, ``lists of sets" $ \SpList \circ \SpSet_{+} $, and ``lists of lists" $ \SpList \circ \SpList_{+} $. 
Sloane and Wieder \cite{sloane2004number-o} call an element of $ (\SpSet \circ \SpList_{+} \circ \SpSet_{+})[n] $ a \textbf{hierarchical ordering} (or \textbf{society}). 
Callan \cite[Section 2]{callan2008sets-jois} gave a bijection between lists of noncrossing sets and sets of lists $ \SpSet \circ \SpList_{+} $. 
Hedmark \cite[Subsection 5.2]{hedmark2017partition} introduced an $ \alpha $-colored partition lattice for a positive integer $ \alpha $ and stated that it can be regarded as $ \SpList_{+}^{\circ \alpha} \circ \SpSet_{+} $. 

\subsection{Tree notation for $ \SpList_{+}^{\circ m} $}
The species $ \SpList_{+}^{\circ m} $ can be considered as the species of ``$ m $-dimensional lists". 
For example $ (((49)(5))((3)(71)(6))((82))) $ is an element of $ \SpList_{+}^{\circ 3}[9] $. 
However, it is difficult to understand the structure of this at first glance. 
Hence we introduce tree notation for $ \SpList_{+}^{\circ m} $. 

The idea is very simple. 
We just regard a list as an ordered rooted tree of height one with labeled leaves, where an \textbf{ordered rooted tree} means a rooted tree whose sibling sets are linearly ordered as lists. 
For example the lists $ (123) $ and $ (2413) $ are expressed as in Figure \ref{fig:trees}. 
\begin{figure}[t]
\centering
\begin{tikzpicture}
\draw (0,1) node[v](r){};
\draw (-0.8,0) node[l](1){1};
\draw (0,0) node[l](2){2};
\draw (0.8,0) node[l](3){3};
\draw (r)--(1);
\draw (r)--(2);
\draw (r)--(3);
\end{tikzpicture}
\hspace{15mm}
\begin{tikzpicture}
\draw (0,1) node[v](r){};
\draw (-1.2,0) node[l](1){2};
\draw (-0.4,0) node[l](2){4};
\draw (0.4,0) node[l](3){3};
\draw (1.2,0) node[l](4){1};
\draw (r)--(1);
\draw (r)--(2);
\draw (r)--(3);
\draw (r)--(4);
\end{tikzpicture}
\caption{$ (123) $ and $ (2431) $}
\label{fig:trees}
\end{figure}

Then the species $ \SpList_{+}^{\circ m} $ can be regarded as a rooted tree of height $ m $. 
For example, 
\begin{align*}
(((49)(5))((3)(71)(6))((82))) \in \SpList_{+}^{\circ 3}[9]
\end{align*}
is expressed as the left in Figure \ref{fig:trees2}. 
We also can express elements of $ \SpList_{+}^{\circ m} \circ \SpSet_{+} $ by taking labels consisting of sets. 
For example 
\begin{align*}
((\{57\}\{3\})(\{149\}\{26\}\{8\})) \in (\SpList_{+}^{\circ 2} \circ \SpSet_{+})[9]
\end{align*}
is as the right in Figure \ref{fig:trees2}. 
\begin{figure}[t]
\centering
\begin{tikzpicture}[scale=.75]
\draw (0,0) node[v](r){};
\draw (-3.6,-3) node[l](c1){4}; 
\draw (-2.8,-3) node[l](c2){9}; 
\draw (-2   ,-3) node[l](c3){5}; 
\draw (-1.2,-3) node[l](c4){3}; 
\draw (-0.4,-3) node[l](c5){7}; 
\draw ( 0.4,-3) node[l](c6){1}; 
\draw ( 1.2,-3) node[l](c7){6}; 
\draw ( 2   ,-3) node[l](c8){8}; 
\draw ( 2.8,-3) node[l](c9){2}; 
\draw (-3.2,-2) node[v](b1){};
\draw (-2   ,-2) node[v](b2){};
\draw (-1.2,-2) node[v](b3){};
\draw ( 0    ,-2) node[v](b4){};
\draw ( 1.2,-2) node[v](b5){};
\draw ( 2.4,-2) node[v](b6){};
\draw (-2.6,-1) node[v](a1){};
\draw ( 0    ,-1) node[v](a2){};
\draw ( 2.4,-1) node[v](a3){};
\draw (r)--(a1);
\draw (r)--(a2);
\draw (r)--(a3);
\draw (a1)--(b1);
\draw (a1)--(b2);
\draw (a2)--(b3);
\draw (a2)--(b4);
\draw (a2)--(b5);
\draw (a3)--(b6);
\draw (b1)--(c1);
\draw (b1)--(c2);
\draw (b2)--(c3);
\draw (b3)--(c4);
\draw (b4)--(c5);
\draw (b4)--(c6);
\draw (b5)--(c7);
\draw (b6)--(c8);
\draw (b6)--(c9);
\end{tikzpicture}
\qquad
\begin{tikzpicture}
\draw (0,0) node[v](r){};
\draw (-2.6,-2) node[l](b1){\{57\}};
\draw (-1.35,-2) node[l](b2){\{3\}};
\draw (0,-2) node[l](b3){\{149\}};
\draw (1.4,-2) node[l](b4){\{26\}};
\draw (2.6,-2) node[l](b5){\{8\}};
\draw (-1.9,-0.9) node[v](a1){};
\draw (1.3,-0.9) node[v](a2){};
\draw (r)--(a1);
\draw (r)--(a2);
\draw (a1)--(b1);
\draw (a1)--(b2);
\draw (a2)--(b3);
\draw (a2)--(b4);
\draw (a2)--(b5);
\end{tikzpicture}
\caption{Elements in $ \SpList_{+}^{\circ 3}[9] $ and $ (\SpList_{+}^{\circ 2} \circ \SpSet_{+})[9] $}
\label{fig:trees2}
\end{figure}

\section{Gain graphs and the associated posets}\label{sec:gain graph}
\subsection{Review of graphic arrangements}
First we recall graphic arrangements and their intersection lattices. 
Let $ \Gamma = (V_{\Gamma}, E_{\Gamma}) $ be a simple graph on vertex set $  V_{\Gamma} = [n] $. 
We can associate $ \Gamma $ with a hyperplane arrangement $ \mathcal{A}_{\Gamma} $ in $ \mathbb{R}^{n} $, called the \textbf{graphic arrangement}, consisting of hyperplanes defined by $ \{x_{i} - x_{j} = 0\} $ with $ \{i,j\} \in E_{\Gamma} $, where $ x_{1}, \dots, x_{\ell} $ denote coordinates of $ \mathbb{R}^{\ell} $. 

It is well known that the intersection lattice $ L(\mathcal{A}_{\Gamma}) $ can be represented by using set partitions as explained below. 
A \textbf{connected partition} of $ \Gamma $ is a set partition of $ V_{\Gamma} $ whose every block induces a connected subgraph of $ \Gamma $. 
Let $ L(\Gamma) $ be the set of all connected partitions of $ \Gamma $ with the partial order defined by refinement. 
Namely, $ \pi \leq \pi^{\prime} $ if each block of $ \pi^{\prime} $ is the union of some blocks of $ \pi $. 
We call $ L(\Gamma) $ the \textbf{lattice of connected partitions} (or the \textbf{lattice of contractions}), which is naturally isomorphic to $ L(\mathcal{A}_{\Gamma}) $. 

Note that the braid arrangement $ \mathcal{B}_{n} $ can be regarded as a graphic arrangement with the complete graph $ K_{n} $. 
Since $ L(K_{n}) $ consists of all set partitions of $ [n] $, the number of flats of the braid arrangement is associated with the Bell number and the Stirling number of the second kind. 

\subsection{Integral gain graphs and affinographic arrangements}
The extended Catalan and Shi arrangements are represented by using gain graphs. 
Gain graphs were introduced by Zaslavsky \cite{zaslavsky1989biased-joctsb} for abstraction of linear independence of hyperplanes of the form $ \{x_{i}-x_{j} = a\} $ with $ a \in \mathbb{Z} $. 
Roughly speaking, a gain graph is a graph with labeled edges $ (i,j,a) $, which corresponds to the $ \{x_{i}-x_{j} = a\} $. 
However, since $ \{x_{i}-x_{j} = a\} = \{x_{j}-x_{i} = -a\} $, we must identify $ (i,j,a) $ with $ (j,i,-a) $. 
Here we give a formal definition of gain graphs which are required in this article. 
See the paper by Zaslavsky \cite{zaslavsky1989biased-joctsb} for a general treatment. 

\begin{definition}
An \textbf{integral gain graph} is a pair $ \Gamma = (V_{\Gamma}, E_{\Gamma}) $ satisfying the following conditions. 
\begin{enumerate}[(i)]
\item $ V_{\Gamma} $ is a finite set. 
\item $ E_{\Gamma} $ is a finite subset of $ \Set{ (u,v, a) \in V_{\Gamma} \times V_{\Gamma} \times \mathbb{Z} | u \neq v} $ divided by the equivalence relation $ \sim $ generated by $ (u,v, a) \sim (v,u, -a) $. 
\end{enumerate}
Let $ \{u,v\}_{a} $ denote the equivalence class containing $ (u,v,a) $. Then $ \{u,v\}_{a} = \{v,u\}_{-a} $. 
Elements in $ V_{\Gamma} $ and $ E_{\Gamma} $ are called vertices and edges of the gain graph $ \Gamma $. 
\end{definition}

\begin{definition}
Suppose that $ \Gamma $ is an integral gain graph on $ [n] $. 
Define an affine arrangement $ \mathcal{A}_{\Gamma} $ in $ \mathbb{R}^{n} $ by 
\begin{align*}
\mathcal{A}_{\Gamma} \coloneqq \Set{ \{x_{i}-x_{j} = a \} | \{i,j\}_{a} \in E_{\Gamma}}. 
\end{align*}
We call $ \mathcal{A}_{\Gamma} $ the \textbf{affinographic arrangement} of $ \Gamma $. 
\end{definition}

Note that every simple graph is regarded as a gain graph by regarding $ \{u, v\} $ as $ \{u,v\}_{0} $. 
So every graphic arrangement is considered to be an affinographic arrangement. 
Hence there is no confusion to use the same symbol $ \mathcal{A}_{\Gamma} $ for the graphic and affinographic arrangements. 

\begin{definition}
Let $ A $ be a finite subset of $ \mathbb{Z} $. 
For every positive integer $ n $, define $ K^{A}_{n} $ as an integral gain graph on $ [n] $ with edges 
\begin{align*}
E_{K^{A}_{n}} \coloneqq \Set{ \{i,j\}_{a} | 1 \leq i < j \leq n, \, a \in A}. 
\end{align*}
We call $ K^{A}_{n} $ the \textbf{complete gain graph with gain $ A $}. 
\end{definition}

For integers $ a,b $ with $ a \leq b $, let $ [a,b]:= \{a, a+1, \dots, b\} \subseteq \mathbb{Z}$. 

\begin{example}
The extended Catalan arrangements and the extended Shi arrangements are obtained by
\begin{align*}
\mathcal{C}^{m}_{n} = \mathcal{A}_{K^{[-m,m]}_{n}} \quad \text{ and } \quad \mathcal{S}^{m}_{n} = \mathcal{A}_{K^{[1-m,m]}_{n}}. 
\end{align*} 
\end{example}

\subsection{Poset of connected partitions}
A \textbf{height function} on a finite set $ B $ is a map $ h \colon B \to \mathbb{Z} $ such that $ \min(h(B)) = 0 $. 

Let $ \Gamma $ be an integral gain graph. 
Consider a pair $ (B,h) $, where $ B \subseteq V_{\Gamma} $ and $ h $ is a height function on $ B $. 
Define $ \Gamma[B, h] $ as the integral gain graph on $ B $ with edges 
\begin{align*}
E_{\Gamma[B,h]} \coloneqq & \Set{ \{u,v\}_{a} \in E_{\Gamma} | u,v \in B, \ h(u)+a=h(v) } \\
 = &\Set{ \{u,v\}_{-h(u)+h(v)} \in E_{\Gamma} | u,v \in B}. 
\end{align*}

Let $ v_{1}, v_{2}, \dots, v_{r} $ be distinct vertices of an integral gain graph $ \Gamma $. 
A \textbf{path} on the vertices $ v_{1}, \dots, v_{r} $ is a set of edges $ \{v_{1},v_{2}\}_{a_{1}}, \{v_{2},v_{3}\}_{a_{2}}, \dots, \{v_{r-1},v_{r}\}_{a_{r-1}} $. 
We say that $ \Gamma $ is \textbf{connected} if there exists a path joining any two distinct vertices. 
A \textbf{connected partition} of $ \Gamma $ is a collection $ \pi = \{(B_{1}, h_{1}), \dots, (B_{k}, h_{k}) \} $ such that $ \{B_{1}, \dots, B_{k}\} $ is a set partition of $ V_{\Gamma} $ and each $ \Gamma[B_{i}, h_{i}] $ is connected. 
An element of $ \pi $ is called a \textbf{block}. 

Let $ \pi $ and $ \pi^{\prime} $ be connected partitions. 
We say that $ \pi $ \textbf{refines} $ \pi^{\prime} $, denoted by $ \pi \leq \pi^{\prime} $ if for any $ (B,h) \in \pi $ there exist $ (B^{\prime}, h^{\prime}) \in \pi^{\prime} $ and $ g \in \mathbb{Z} $ such that $ B \subseteq B^{\prime} $ and $ h(v) = g + h^{\prime}(v) $ for any $ v \in B $. 
Let $ L(\Gamma) $ denote the set of all connected partitions, which forms a poset together with the refinement.
Call $ L(\Gamma) $ the \textbf{poset of connected partitions} of $ \Gamma $. 

Note that when a simple graph $ \Gamma $ is viewed as a gain graph, connected partitions coincide with the usual connected partitions of the simple graph $ \Gamma $.
Therefore the poset of connected partitions of a gain graph is a generalization of the lattice of connected partitions of a simple graph. 

\begin{remark}
For a general gain graph with gain group $ G $, we need notion of $ G $-labeled set to define connected partitions. 
However, for integral gain graphs, we need only height functions, which were introduced by Corteel, Forge, and Ventos \cite{corteel2015bijections-ejoc}. 
\end{remark}

\begin{theorem}\label{part-arr}
Let $ \Gamma $ be an integral gain graph on $ [n] $. 
\begin{enumerate}[(1)]
\item\label{part-arr 1} If $ (B,h) $ is a block of a connected partition of $ \Gamma $ with $ B = \{i_{1}, \dots, i_{r}\} $, then 
\begin{align*}
X_{(B,h)} \coloneqq \bigcap_{\{i,j\}_{a} \in E_{\Gamma[B,h]}}\{x_{i}-x_{j}=a\}
= \{x_{i_{1}}+h(i_{1}) = \cdots = x_{i_{r}}+h(i_{r}) \}. 
\end{align*}
\item\label{part-arr 2} The following map is an isomorphism of posets. 
\begin{align*}
\begin{array}{rcl}
L(\Gamma) & \longrightarrow & L(\mathcal{A}_{\Gamma}) \\
\pi & \longmapsto & \displaystyle\bigcap_{(B,h) \in \pi} X_{(B,h)}. 
\end{array}
\end{align*}
\item\label{part-arr 3} For any nonnegative integer $ k $, the isomorphism in (\ref{part-arr 2}) induces a bijection between $ L_{k}(\Gamma) $ and $ L_{k}(\mathcal{A}_{\Gamma}) $, where 
\begin{align*}
L_{k}(\Gamma) \coloneqq \Set{\pi \in L(\Gamma) | \ |\pi| = k}
\end{align*}
and $ |\pi| $ denotes the number of blocks of $ \pi $. 
\end{enumerate} 
\end{theorem}
\begin{proof}
(\ref{part-arr 1}) 
Recall $ E_{\Gamma[B,h]} = \Set{ \{i,j\}_{-h(i)+h(j)} \in E_{\Gamma} | i,j \in B} $. 
The hyperplane corresponding to the edge $ \{i,j\}_{-h(i)+h(j)} $ is 
\begin{align*}
\{x_{i}-x_{j} = -h(i)+h(j) \} = \{x_{i}+h(i) = x_{j}+h(j) \}. 
\end{align*}
Since $ \Gamma[B,h] $ is connected, the intersection coincides with the right hand side. 

(\ref{part-arr 2}) 
This follows from theorems by Zaslavsky {\cite[Corollary 4.5(a)]{zaslavsky2003biased-joctsb}}, {\cite[Lemma 3.1A and 3.1B]{zaslavsky1985geometric}}. 

(\ref{part-arr 3})
Obvious from (\ref{part-arr 2}). 
\end{proof}

\subsection{Partitional decompositions}
Let $ A $ be a finite subset of $ \mathbb{Z} $. 
The construction of the complete gain graph with gain $ A $ is considered to be an $ \FinLinSetL $-species. 
Namely, for each finite linearly ordered set $ V $, we may construct the complete gain graph $ K_{V}^{A} $ on $ V $ with gain $ A $.
Hence, there exist $ \FinLinSetL $-species $ \SpLK{}{A} $ and $ \SpLK{k}{A} $ such that $ \SpLK{}{A}[n] = L(K^{A}_{n}) $ and $ \SpLK{k}{A}[n] = L_{k}(K^{A}_{n}) $. 
In particular, if $ -A = A $, then the $ \FinLinSetL $-species $ \SpLK{}{A} $ and $ \SpLK{k}{A} $ may be considered to be species. 
\begin{lemma}\label{part decomp}
Let $ A $ be a finite subset of $ \mathbb{Z} $. 
Then 
\begin{align*}
\SpLK{k}{A} = \SpSet_{k} \circ \SpLK{1}{A} \quad \text{ and } \quad \SpLK{}{A} = \SpSet \circ \SpLK{1}{A}
\end{align*}
In particular, 
\begin{align*}
\SpLCat{k}{m} &= \SpLK{k}{[-m,m]} = \SpSet_{k} \circ \SpLK{1}{[-m.m]}, 
\qquad
\SpLCat{}{m} = \SpLK{}{[-m,m]} = \SpSet_{} \circ \SpLK{1}{[-m.m]}, 
\end{align*}
\begin{align*}
\SpLShi{k}{m} = \SpLK{k}{[1-m,m]} = \SpSet_{k} \circ \SpLK{1}{[1-m.m]}, 
\qquad
\SpLShi{}{m} = \SpLK{}{[1-m,m]} = \SpSet_{} \circ \SpLK{1}{[1-m.m]}. 
\end{align*}
\end{lemma}
\begin{proof}
Let $ (B,h) $ be a block of a connected partition of $ K^{A}_{n} $. 
Then the block $ (B,h) $ may be regarded as an element of $ \SpLK{1}{A}[B] $. 
This identification leads to the natural isomorphisms $ \SpLK{k}{A} = \SpSet_{k} \circ \SpLK{1}{A} $ and $ \SpLK{}{A} = \SpSet \circ \SpLK{1}{A} $. 
The rest follows by Theorem \ref{part-arr}. 
\end{proof}

\section{Proofs}\label{sec:proofs}

\subsection{Proof of Theorem \ref{main catalan}}
In order to describe $ \SpLCat{}{m} $, it is sufficient to determine the species $ \SpLK{1}{[-m,m]} $ by Lemma \ref{part decomp}. 
Given a finite set $ V $, an element of $ \SpLK{1}{[-m,m]}[V] = L_{1}(K_{V}^{[-m,m]}) $ is identified with a height function $ h $ on $ V $ such that $ K_{V}^{[-m,m]}[V,h] $ is connected. 
In order to characterize such functions, define $ \gap(h) $ and $ \level(h) $ for each height functions as follows. 

Suppose that $ h(V) = \{a_{1}, \dots, a_{k}\} $ with $ 0=a_{1} < a_{2} < \dots < a_{k} $. 
Then we define 
\begin{align*}
\level(h) &\coloneqq \left(h^{-1}(a_{1}), h^{-1}(a_{2}), \dots, h^{-1}(a_{k})\right), \\
\gap(h) &\coloneqq (a_{2}-a_{1}, a_{3}-a_{2}, \dots, a_{k}-a_{k-1}).  
\end{align*}
Let $ \max(\gap(h)) $ denote the maximum of entries of $ \gap(h) $. 

\begin{lemma}\label{connectedness Catalan}
Let $ V $ be a finite set and $ h $ a height function on $ V $. 
For every nonnegative integer $ m $, $ K_{V}^{[-m,m]}[V,h] $ is connected if and only if $ \max(\gap(h)) \leq m $. 
\end{lemma}
\begin{proof}
Let $ \level(h) = (B_{1}, \dots, B_{k}) $ and $ \gap(h) = (\alpha_{1}, \dots, \alpha_{k-1}) $. 
First suppose that $ K_{V}^{[-m,m]}[V,h] $ is connected. 
Then for every $ i \in \{1, \dots, k-1\} $ there exists an edge $ \{u,v\}_{a} $, where $ u \in B_{1} \cup \dots \cup B_{i} $ and $ v \in B_{i+1} \cup \dots \cup B_{k} $. 
Then $ \alpha_{i} \leq h(v)-h(u) = a \leq m $. 
Thus $ \max(\gap(h)) \leq m $. 

To prove the converse, suppose that $ \max(\gap(h)) \leq m $. 
Suppose that $ u, v \in B_{i} $ for some $ i \in \{1, \dots, k\} $. 
Then $ \{u,v\}_{0} $ is an edge of $ K_{V}^{[-m,m]}[V,h] $. 
Now, let $ i \in \{1, \dots, k-1\} $ and take vertices $ u \in B_{i} $ and $ v \in B_{i+1} $. 
Then $ h(v)-h(u) = \alpha_{i} \leq \max(\gap(h)) \leq m $. 
Therefore $ \{u, v\}_{\alpha_{i}} $ is an edge of $ K_{V}^{[-m,m]}[V,h] $. 
Hence we can deduce $ K_{V}^{[-m,m]}[V,h] $ is connected. 
\end{proof}

\begin{lemma}\label{decomposition catalan}
Let $ m $ be a nonnegative integer. 
Then $ \SpLK{1}{[-m,m]} = \SpList_{+}^{\circ m} \circ \SpSet_{+} $. 
\end{lemma}
\begin{proof}
We proceed by induction on $ m $. 
First suppose that $ m = 0 $. 
Let $ V $ be a finite set. 
Then $ K_{V}^{\{0\}}[V,h] $ is connected if and only if $ h $ is identically $ 0 $ by Lemma \ref{connectedness Catalan}. 
Hence the identification $ (V,0) $ with $ V $ yields the natural isomorphism $ \SpLK{1}{\{0\}} = \SpSet_{+} $. 

Now suppose that $ m \geq 1 $. 
We will prove $ \SpLK{1}{[-m,m]} = \SpList_{+} \circ \SpLK{1}{[-(m-1),m-1]} $. 
Define a natural transformation $ \eta \colon \SpLK{1}{[-m,m]} \to \SpList_{+} \circ \SpLK{1}{[-(m-1),m-1]} $ as follows. 
Let $ V $ be a finite set and $ h $ be a height function on $ V $ such that $ K_{V}^{[-m,m]}[V,h] $ is connected. 
Let $ \level(h) = (B_{1}, \dots, B_{k}) $ and $ \gap(h) = (\alpha_{1}, \dots, \alpha_{k-1}) $ and suppose that $ \Set{j | \alpha_{j} = m} = \{j_{1}, \dots, j_{s-1}\} $ with $ j_{1} < \dots < j_{s-1} $. 
For each $ i \in \{1, \dots, s\} $ define $ C_{i} \coloneqq B_{j_{i-1}+1} \cup B_{j_{i-1}+2} \cup \dots \cup B_{j_{i}} $, where $ j_{0} \coloneqq 0 $ and $ j_{s} \coloneqq k $. 
Define the height function $ h_{i} $ on $ C_{i} $ by $ h_{i}(v) \coloneqq h(v) - \min(h(C_{i})) \quad (v \in C_{i}) $. 
Then $ \max(\gap(h_{i})) \leq m-1 $ and hence $ (C_{i}, h_{i}) $ is an element of $ \SpLK{1}{[-(m-1),m-1]}[C_{i}] $. 
Therefore we define $ \eta_{V} $ by $ \eta_{V}\left( (V,h) \right) \coloneqq \left( (C_{1}, h_{1}), \dots, (C_{s}, h_{s}) \right) $. 

Next we will construct another natural transformation $ \xi \colon \SpList_{+} \circ \SpLK{1}{[-(m-1),m-1]} \to \SpLK{1}{[-m,m]} $ as follows. 
Let $ V $ be a finite set and take an element $ \left( (C_{1}, h_{1}), \dots, (C_{s}, h_{s}) \right) $ from $ \SpList_{+} \circ \SpLK{1}{[-(m-1),m-1]} $. 
Then $ \{C_{1}, \dots, C_{s}\} $ is a set partition of $ V $ and $ (C_{i}, h_{i}) \in \SpLK{1}{[-(m-1),m-1]}[C_{i}] $ for every $ i \in \{1, \dots, s\} $. 
By Lemma \ref{connectedness Catalan} $ \max(\gap(h_{i})) \leq m-1 $. 
Let $ h $ be the height function on $ V $ defined by 
$ h(v) \coloneqq h_{i}(v) + \sum_{j = 1}^{i-1}\left( m + \max(h_{j}(C_{j})) \right) $ for $
v \in C_{i}.  $
Since $ \max(\gap(h)) \leq m $, $ (V,h) \in \SpLK{1}{[-m,m]}[V] $. 
Thus we may define $ \xi_{V} $ by $ \xi_{V}\left( \left( (C_{1}, h_{1}), \dots, (C_{s}, h_{s}) \right) \right) \coloneqq (V,h) $. 

It is obvious that $ \eta $ and $ \xi $ are inverse to each other. 
Therefore $ \SpLK{1}{[-m,m]} = \SpList_{+} \circ \SpLK{1}{[-(m-1),m-1]} $. 
By induction hypothesis, we can conclude that $ \SpLK{1}{[-m,m]} = \SpList_{+}^{\circ m} \circ \SpSet_{+} $. 
\end{proof}

\begin{proof}[Proof of Theorem \ref{main catalan}]
Use Lemma \ref{part decomp} and Lemma \ref{decomposition catalan}. 
\end{proof}

\begin{example}\label{ex:catalan}
Consider $ ((B_{1}B_{2})(B_{3}B_{4}B_{5})) = ((\{57\}\{3\})(\{149\}\{26\}\{8\})) \in (\SpList_{+}^{\circ 2} \circ \SpSet_{+})[9] $ (See Figure \ref{fig:example catalan}). 
\begin{figure}[t]
\centering
\begin{tikzpicture}
\draw (0,0) node[v](r){};
\draw (-2.6,-2) node[l](b1){\{57\}};
\draw (-1.35,-2) node[l](b2){\{3\}};
\draw (0,-2) node[l](b3){\{149\}};
\draw (1.4,-2) node[l](b4){\{26\}};
\draw (2.6,-2) node[l](b5){\{8\}};
\draw (-1.9,-0.9) node[v](a1){};
\draw (1.3,-0.9) node[v](a2){};
\draw (r)--(a1);
\draw (r)--(a2);
\draw (a1)--(b1);
\draw (a1)--(b2);
\draw (a2)--(b3);
\draw (a2)--(b4);
\draw (a2)--(b5);
\draw (-3.8,-2.7) node{$ \alpha $};
\draw (-1.9,-2.7) node{$ 1 $};
\draw (-0.75,-2.7) node{$ 2 $};
\draw (0.76,-2.7) node{$ 1 $};
\draw (2.1,-2.7) node{$ 1 $};
\draw (-3.8,-3.4) node{$ h $};
\draw (-2.6,-3.4) node{$ 0 $};
\draw (-1.35,-3.4) node{$ 1 $};
\draw (0,-3.4) node{$ 3 $};
\draw (1.4,-3.4) node{$ 4 $};
\draw (2.6,-3.4) node{$ 5 $};
\end{tikzpicture}
\caption{An example of the correspondence for the extended Catalan arrangement}
\label{fig:example catalan}
\end{figure}
We construct the corresponding flat of the extended Catalan arrangement $ \mathcal{C}^{2}_{9} $. 
First, for each $ i $, let $ \alpha_{i} $ denote the height of the minimal tree containing leaves $ B_{i} $ and $ B_{i+1} $. 
In this case we have the integer composition $ \alpha = (\alpha_{1}, \alpha_{2}, \alpha_{3}, \alpha_{4}) = (1,2,1,1) $. 
By taking the partial sum $\sum_{j=1}^{i-1}\alpha_{j}$ for each $i$, we obtain the sequence of heights $ (0,1,3,4,5) $. 
The height function $ h:[9]\rightarrow \mathbb{Z} $ is obtained by the following table. 
\begin{align*}
\begin{array}{c|ccccccccc}
v & 5 & 7 & 3 & 1 & 4 & 9 & 2 & 6  & 8 \\
\hline
h(v) & 0 & 0 & 1 & 3 & 3 & 3 & 4 & 4 & 5
\end{array}
\end{align*}
The corresponding flat in $ \mathcal{C}^{2}_{9} $ is 
\begin{align*}
\{x_{5}=x_{7}=x_{3}+1=x_{1}+3=x_{4}+3=x_{9}+3=x_{2}+4=x_{6}+4=x_{8}+5\}. 
\end{align*}
\end{example}

\subsection{Proof of Theorem \ref{main shi}}
We assume that all species in this subsection are $ \FinLinSetL $-species and $ V $ denotes a finite linearly ordered set. 
The symbols $ \min (B) $ and $ \max (B) $ stand for the minimum and the maximum elements of a subset $ B \subseteq V $. 

\begin{lemma}\label{connectedness Shi}
Let $ V $ be a finite linearly ordered set and $ h $ a height function on $ V $. 
Let $ \level(h) = (B_{1}, \dots, B_{k}) $ and $ \gap(h) = (\alpha_{1}, \dots, \alpha_{k-1}) $. 
For every positive integer $ m $, $ K_{V}^{[1-m,m]}[V,h] $ is connected if and only if $ \max(\gap(h)) \leq m $ holds and $ \alpha_{i} = m $ implies $ \min(B_{i}) < \max(B_{i+1}) $. 
\end{lemma}
\begin{proof}
Suppose that $ K_{V}^{[1-m,m]}[V,h] $ is connected. 
Then $ K_{V}^{[-m,m]}[V,h] $ is also connected since $ K_{V}^{[1-m,m]}[V,h] $ is a subgraph of $ K_{V}^{[-m,m]}[V,h] $. 
Therefore $ \max(\gap(h)) \leq m $ by Lemma \ref{connectedness Catalan}. 
Suppose that $ \alpha_{i} = m $. 
Then there exist vertices $ u \in B_{i} $ and $ v \in B_{i+1} $ such that $ \{u,v\}_{m} $ is an edge of $ K_{V}^{[1-m,m]}[V,h] $ since $ K_{V}^{[1-m,m]}[V,h] $ is connected, which implies $ u < v $ in $ V $. 
Thus it follows that $ \min(B_{i}) < \max(B_{i+1}) $. 
The proof of the converse is similar to the proof of Lemma \ref{connectedness Catalan}. 
\end{proof}

\begin{lemma}\label{decomposition shi}
Let $ m $ be a positive integer. 
Then $ \SpLK{1}{[1-m,m]} = \SpList_{+}^{\circ m} $. 
\end{lemma}
\begin{proof}
We proceed by induction on $ m $. 
Suppose that $ m=1 $. 
We will construct a natural transformation $ \eta \colon \SpLK{1}{[0,1]} \to \SpList_{+} $. 
Let $ V $ be a finite linearly ordered set and $ h $ a height function on $ V $ such that $ K_{V}^{[0,1]}[V,h] $ is connected.
Let $ \level(h) = (B_{1}, \dots, B_{k}) $. 
Make the list $ \beta_{i} = (v_{i1}, \dots, v_{ij_{i}}) $ such that $ B_{i} = \{v_{i1}, \dots, v_{ij_{i}}\} $ with $ v_{i1} > \dots > v_{ij_{i}} $. 
Define $ \eta_{V} $ as the list obtaining by concatenating $ \beta_{1}, \dots, \beta_{k} $. 

Next we will construct another natural transformation $ \xi \colon \SpList_{+} \to \SpLK{1}{[0,1]} $. 
Every list in $ \SpList_{+}[V] $ can be expressed as $ (v_{11}, \dots, v_{1j_{1}}, v_{21}, \dots, v_{2j_{2}}, \dots, v_{k1}, \dots, v_{kj_{k}}) $, where $ v_{i1} > \dots > v_{ij_{i}} $ and $  v_{ij_{i}} < v_{i+1 \, 1} $ for each $ i \in \{1, \dots k-1\} $. 
Let $ B_{i} \coloneqq \{ v_{i1}, \dots, v_{ij_{i}} \} $ and define the height function on $ V $ by $ h(v) \coloneqq i-1 $ for $ v \in B_{i} $. 
Then $ \max(\gap(h)) \leq 1 $ and $ \min(B_{i}) = v_{ij_{i}} < v_{i+1 \, 1} = \max(B_{i+1}) $. 
Therefore by Lemma \ref{connectedness Shi} $ (V,h) \in \SpLK{1}{[0,1]}[V] $. 
Thus we may define $ \xi_{V} $ by the correspondence above. 
One can show that $ \eta $ and $ \xi $ are inverse to each other and hence $ \SpLK{1}{[0,1]} = \SpList_{+} $. 

Assume that $ m \geq 2 $ and we will prove that $ \SpLK{1}{[1-m,m]} = \SpList_{+} \circ \SpLK{1}{[2-m, m-1]} $. 
We will construct a natural transformation $ \eta \colon \SpLK{1}{[1-m, m]} \to \SpList_{+} \circ \SpLK{1}{[2-m,m-1]} $.
Let $ h $ be a height function on a finite linearly ordered set $ V $ such that $ K_{V}^{[1-m,m]}[V,h] $ is connected. 
Let $ \level(h) = (B_{1}, \dots, B_{k}) $ and $ \gap(h) = (\alpha_{1}, \dots, \alpha_{k-1}) $. 
Suppose that $ \Set{j | \alpha_{j} = m} \cup \Set{j | \alpha_{j} = m-1, \ \min(B_{j}) > \max(B_{j+1}) } = \{j_{1}, \dots, j_{s-1}\} $ with $ j_{1} < \dots < j_{s-1} $. 
For each $ i \in \{1, \dots, s\} $, put $ C_{i} \coloneqq B_{j_{i-1}+1} \cup \dots \cup B_{j_{i}} $, where $ j_{0} = 0 $ and $ j_{s} \coloneqq k $. 
Define the height function $ h_{i} $ on $ C_{i} $ by $ h_{i}(v) \coloneqq h(v) - \min(h(C_{i})) $ for $ v \in C_{i} $. 
Then $ \max(\gap(h_{i})) \leq m-1 $ and $ \level(h_{i}) = (B_{j_{i-1}+1}, \dots, B_{j_{i}}) $. 
Moreover for $ i \in \{j_{i-1}+1, \dots, j_{i}\} $, if $ \alpha_{i} = m-1 $, then $ \min(B_{i}) < \max(B_{i+1}) $. 
Therefore $ (C_{i}, h_{i}) \in \SpLK{1}{[2-m,m-1]}[C_{i}] $ by Lemma \ref{connectedness Shi}. 
Thus we may define $ \eta_{V} $ by $ \eta_{V}((V,h)) \coloneqq ((C_{1}, h_{1}), \dots, (C_{s}, h_{s})) $. 

To construct another natural transformation $ \xi \colon \SpList_{+} \circ \SpLK{1}{[2-m,m-1]} \to \SpLK{1}{[1-m,m]} $, take an element $ ((C_{1}, h_{1}), \dots, (C_{s}, h_{s})) \in \SpList_{+} \circ \SpLK{1}{[2-m,m-1]}[V] $. 
For each $ i \in \{1, \dots, s-1\} $, define the integer $ \mu_{i} $ as follows. 
If the minimum of the terminal block of $ \level(h_{i}) $ is less than the maximum of the initial block of $ \level(h_{i+1}) $, then $ \mu_{i} \coloneqq m $. 
Otherwise let $ \mu_{i} \coloneqq m-1 $. 
Define the height function $ h $ on $ V $ by $ h(v) \coloneqq h_{i}(v) + \sum_{j=1}^{i-1}\left(\mu_{j}+\max(h_{j}(C_{j}))\right) $ for $ v \in C_{i} $. 
Then one can deduce that $ (V,h) \in \SpLK{1}{[1-m,m]}[V] $ by Lemma \ref{connectedness Shi}. 
Hence we may define $ \xi_{V} $ by $ \xi_{V}(((C_{1}, h_{1}), \dots, (C_{s}, h_{s}))) \coloneqq (V,h) $. 

It is easy to show that $ \eta $ and $ \xi $ are inverse to each other and hence $ \SpLK{1}{[1-m,m]} = \SpList_{+} \circ \SpLK{1}{[2-m,m-1]} $. 
Finally by the induction hypothesis, we have $ \SpLK{1}{[1-m,m]} = \SpList_{+}^{\circ m} $. 
\end{proof}

\begin{proof}[Proof of Theorem \ref{main shi}]
Use Lemma \ref{part decomp} and Lemma \ref{decomposition shi}. 
\end{proof}

\begin{example}
Consider $ (((v_{1}v_{2})(v_{3}))((v_{4})(v_{5}v_{6})(v_{7}))((v_{8}v_{9}))) = (((49)(5))((3)(71)(6))((82))) \in \SpList_{+}^{\circ 3}[9] $ (See Figure \ref{fig:example shi}). 
\begin{figure}[t]
\centering
\begin{tikzpicture}[scale=.75]
\draw (0,0) node[v](r){};
\draw (-3.6,-3) node[l](c1){4}; 
\draw (-2.8,-3) node[l](c2){9}; 
\draw (-2   ,-3) node[l](c3){5}; 
\draw (-1.2,-3) node[l](c4){3}; 
\draw (-0.4,-3) node[l](c5){7}; 
\draw ( 0.4,-3) node[l](c6){1}; 
\draw ( 1.2,-3) node[l](c7){6}; 
\draw ( 2   ,-3) node[l](c8){8}; 
\draw ( 2.8,-3) node[l](c9){2}; 
\draw (-3.2,-2) node[v](b1){};
\draw (-2   ,-2) node[v](b2){};
\draw (-1.2,-2) node[v](b3){};
\draw ( 0    ,-2) node[v](b4){};
\draw ( 1.2,-2) node[v](b5){};
\draw ( 2.4,-2) node[v](b6){};
\draw (-2.6,-1) node[v](a1){};
\draw ( 0    ,-1) node[v](a2){};
\draw ( 2.4,-1) node[v](a3){};
\draw (r)--(a1);
\draw (r)--(a2);
\draw (r)--(a3);
\draw (a1)--(b1);
\draw (a1)--(b2);
\draw (a2)--(b3);
\draw (a2)--(b4);
\draw (a2)--(b5);
\draw (a3)--(b6);
\draw (b1)--(c1);
\draw (b1)--(c2);
\draw (b2)--(c3);
\draw (b3)--(c4);
\draw (b4)--(c5);
\draw (b4)--(c6);
\draw (b5)--(c7);
\draw (b6)--(c8);
\draw (b6)--(c9);
\draw (-5,-3.8) node{$ \alpha $}; 
\draw (-3.2,-3.8) node{$ 1 $}; 
\draw (-2.4,-3.8) node{$ 2 $}; 
\draw (-1.6,-3.8) node{$ 3 $}; 
\draw (-0.8,-3.8) node{$ 2 $}; 
\draw (0,-3.8) node{$ 1 $}; 
\draw (0.8,-3.8) node{$ 2 $}; 
\draw (1.6,-3.8) node{$ 3 $}; 
\draw (2.4,-3.8) node{$ 1 $}; 
\draw (-5,-4.6) node{$ \alpha^{\prime} $}; 
\draw (-3.2,-4.6) node{$ 1 $}; 
\draw (-2.4,-4.6) node{$ 1 $}; 
\draw (-1.6,-4.6) node{$ 2 $}; 
\draw (-0.8,-4.6) node{$ 2 $}; 
\draw (0,-4.6) node{$ 0 $}; 
\draw (0.8,-4.6) node{$ 2 $}; 
\draw (1.6,-4.6) node{$ 3 $}; 
\draw (2.4,-4.6) node{$ 0 $}; 
\draw (-5,-5.4) node{$ h $}; 
\draw (-3.6,-5.4) node{$ 0 $}; 
\draw (-2.8,-5.4) node{$ 1 $}; 
\draw (-2,-5.4) node{$ 2 $}; 
\draw (-1.2,-5.4) node{$ 4 $}; 
\draw (-0.4,-5.4) node{$ 6 $}; 
\draw (0.4,-5.4) node{$ 6 $}; 
\draw (1.2,-5.4) node{$ 8 $}; 
\draw (2,-5.4) node{$ 11 $}; 
\draw (2.8,-5.4) node{$ 11 $}; 
\end{tikzpicture}
\caption{An example of the correspondence for the extended Shi arrangement}
\label{fig:example shi}
\end{figure}
We construct the corresponding flat of the extended Shi arrangement $ \mathcal{S}^{3}_{9} $. 
First let $ \alpha $ be the integer composition obtained in a similar way in Example \ref{ex:catalan}. 
In this case $ \alpha = (1,2,3,2,1,2,3,1) $. 
Next define the integer composition $ \alpha^{\prime} $ by 
\begin{align*}
\alpha^{\prime}_{i} \coloneqq \begin{cases}
\alpha_{i} & \text{ if } v_{i} < v_{i+1}, \\
\alpha_{i}-1 & \text{ if } v_{i} > v_{i+1}. 
\end{cases}
\end{align*}
In this case $ \alpha^{\prime} = (1,1,2,2,0,2,3,0) $. 
By taking the partial sum $\sum_{j=1}^{i-1}\alpha_{j}$ for each $i$, we obtain the sequence of heights $ (0,1,2,4,6,6,8,11,11) $. 
The height function $ h $ is obtained by the following table.
\begin{align*}
\begin{array}{c|ccccccccc}
v & 4 & 9 & 5 & 3 & 7 & 1 & 6 & 8  & 2 \\
\hline
h(v) & 0 & 1 & 2 & 4 & 6 & 6 & 8 & 11 & 11
\end{array}
\end{align*}
The corresponding flat is 
\begin{align*}
\{x_{4}=x_{9}+1=x_{5}+2=x_{3}+4=x_{7}+6=x_{1}+6=x_{6}+8=x_{8}+11=x_{2}+11\}. 
\end{align*}
\end{example}

\subsection{Proof of Theorem \ref{main matrices}}
Let $ \SpF $ be a species with $ \SpF[\varnothing] = \varnothing $. 
We consider the infinite matrix $ \Big[ \big| (\SpSet_{i} \circ \SpF)[j] \big| \Big] $. 
Note that almost all entries of each column of the matrix are $ 0 $ since 
\begin{align*}
\sum_{i=1}^{\infty} \big|(\SpSet_{i} \circ \SpF)[j]\big|=\big|(\SpSet \circ \SpF)[j]\big| < \infty. 
\end{align*}
We show that substitution of species is compatible with product of the infinite matrices. 
\begin{proposition}\label{matrix product}
Let $ \SpF $ and $ \SpG $ be species with $ \SpF[\varnothing] = \SpG[\varnothing] = \varnothing $. 
Then 
\begin{align*}
\Big[ \big| (\SpSet_{i} \circ \SpF)[j] \big| \Big] \Big[ \big| (\SpSet_{i} \circ \SpG)[j] \big| \Big] = \Big[ \big| (\SpSet_{i} \circ \SpF \circ \SpG)[j] \big| \Big]. 
\end{align*}
\end{proposition}
\begin{proof}
Fix a positive integer $ i $. 
By definition,
\begin{align*}
(\SpSet_{i} \circ \SpF \circ \SpG)(x) = \sum_{j=0}^{\infty} \big|(\SpSet_{i} \circ \SpF \circ \SpG)[j]\big|\, \frac{x^{j}}{j!}. 
\end{align*}
We give another calculation of the series as follows. 
\begin{align*}
(\SpSet_{i} \circ \SpF \circ \SpG)(x) 
&= (\SpSet_{i} \circ \SpF)(\SpG(x)) \\
&= \sum_{k=0}^{\infty}\big| (\SpSet_{i} \circ \SpF)[k] \big| \, \frac{\SpG(x)^{k}}{k!} \\
&= \sum_{k=0}^{\infty}\big| (\SpSet_{i} \circ \SpF)[k] \big| \, (\SpSet_{k} \circ \SpG)(x) \\
&= \sum_{k=0}^{\infty}\big| (\SpSet_{i} \circ \SpF)[k] \big| \, \sum_{j=0}^{\infty}\big| (\SpSet_{k} \circ \SpG)[j] \big| \, \frac{x^{j}}{j!} \\
&= \sum_{j=0}^{\infty} \left( \sum_{k=0}^{\infty} \big| (\SpSet_{i} \circ \SpF)[k] \big| \, \big| \SpSet_{k} \circ \SpG[j] \big| \right) \, \frac{x^{j}}{j!}. 
\end{align*}
Therefore we have 
\begin{align*}
\big|\SpSet_{i} \circ \SpF \circ \SpG[j]\big| = \sum_{k=0}^{\infty} \big| \SpSet_{i} \circ \SpF[k] \big| \, \big| \SpSet_{k} \circ \SpG[j] \big| 
\end{align*}
for any positive integers $ i $ and $ j $. 
Hence the assertion holds. 
\end{proof}

\begin{example}\label{Lah matrix}
Let $ \SpCycle $ denote the \textbf{species of cyclic permutations}. 
Then the substitution $ \SpSet \circ \SpCycle_{+} $ coincides with the \textbf{species of permutations}. 
As mentioned in \cite[p. 346]{bergeron1998combinatorial}, we have that $ \SpList = \SpSet \circ \SpCycle_{+} $ as $ \FinLinSetL $-species. 
Indeed let each permutation $ \sigma \in (\SpSet \circ \SpCycle_{+})[n] $ correspond to the list $ (\sigma(1), \sigma(2), \cdots, \sigma(n)) \in \SpList[n] $. 
It is easy to see that this correspondence is bijective. 
Note that $ \SpList $ and $ \SpSet \circ \SpCycle_{+} $ are not isomorphic as $ \FinSetB $-species. 
Recall that $ c, S $, and $\Big[ \big| (\SpSet_{i} \circ \SpList_{+})[j] \big| \Big] $ is the infinite upper triangular matrix consisting of Stirling numbers of the first and second kind, and Lah numbers. 
We can recover the following well-known equality. 
\begin{align*}
\Big[ \big| (\SpSet_{i} \circ \SpList_{+})[j] \big| \Big] &= \Big[ \big| (\SpSet_{i} \circ \SpSet_{+} \circ \SpCycle_{+})[j] \big| \Big] = \Big[ \big| (\SpSet_{i} \circ \SpSet_{+})[j] \big| \Big] \Big[ \big| (\SpSet_{i} \circ \SpCycle_{+})[j] \big| \Big] = Sc. 
\end{align*}
\end{example}

\begin{proof}[Proof of Theorem \ref{main matrices}]
From Theorem \ref{main catalan}, Proposition \ref{matrix product}, and Example \ref{Lah matrix} 
\begin{align*}
\Big[ \, \big|L_{i}(\mathcal{C}^{m}_{j})\big| \, \Big] 
&= \Big[ \, \big| \SpLCat{i}{m}[j] \big| \, \Big]
= \Big[ \, \big| (\SpSet_{i} \circ \SpList_{+}^{\circ m} \circ \SpSet_{+})[j] \big| \, \Big] \\
&= \Big[ \big| (\SpSet_{i} \circ \SpList_{+})[j] \big| \Big]^{m} \Big[ \big| (\SpSet_{i} \circ \SpSet_{+})[j] \big| \Big]
= (Sc)^{m}S. 
\end{align*}
Similarly from Theorem \ref{main shi}, Proposition \ref{matrix product}, and Example \ref{Lah matrix} 
\begin{align*}
\Big[ \, \big|L_{i}(\mathcal{S}^{m}_{j})\big| \, \Big] 
&= \Big[ \, \big| \SpLShi{i}{m}[j] \big| \, \Big]
= \Big[ \, \big| (\SpSet_{i} \circ \SpList_{+}^{\circ m})[j] \big| \, \Big] \\
&= \Big[ \big| (\SpSet_{i} \circ \SpList_{+})[j] \big| \Big]^{m} 
= (Sc)^{m}. 
\end{align*}
\end{proof}

\subsection{Proof of Theorem \ref{main shilah}}
Let $ a_{ij} $ denote the Lah number $ \big| (\SpSet_{i} \circ \SpList_{+})[j] \big| $. 
By Proposition \ref{lah number}, $ a_{ij} = \frac{j!(j-1)!}{i!(i-1)!(j-i)!} $. 
Note that if $ i > j $, then $ a_{ij}=0 $. 
\begin{lemma}\label{lah power}
For every positive integer $ m $, $ [a_{ij}]^{m} = [m^{j-i}a_{ij}] $. 
\end{lemma}
\begin{proof}
We proceed by induction on $ m $. 
If $ m=1 $, then it is trivial. 
Assume that $ m \geq 2 $. 
By the induction hypothesis, the $ (i,j) $-entry of the matrix $ [a_{ij}]^{m} = [a_{ij}][a_{ij}]^{m-1} $ is 
\begin{align*}
\sum_{k=i}^{j}a_{ik} (m-1)^{j-k}a_{kj} &= \sum_{k=i}^{j} \dfrac{k!(k-1)!}{i!(i-1)!(k-i)!}(m-1)^{j-k}\dfrac{j!(j-1)!}{k!(k-1)!(j-k)!} \\
&= \dfrac{j!(j-1)!}{i!(i-1)!(j-i)!}\sum_{k=i}^{j}\dfrac{(j-i)!}{(j-k)!(k-i)!}(m-1)^{j-k}\\
&= a_{ij}\sum_{k=i}^{j}\binom{j-i}{k-i}(m-1)^{j-k} 
= a_{ij}\sum_{k=0}^{j-i}\binom{j-i}{k}(m-1)^{j-i-k}  = m^{j-i}a_{ij}.
\end{align*}
This completes the proof. 
\end{proof}
\begin{proof}[Proof of Theorem \ref{main shilah}]
The assertion holds immediately from Proposition \ref{lah number}, Theorem \ref{main matrices} and Lemma \ref{lah power}. 
\end{proof}

\section*{Acknowledgments}
This work was supported by JSPS KAKENHI Grant Number JP20K14282.

\newpage
\appendix
\section{Numerical tables}\label{sec:table}

\begin{table}[H]
\centering
\begin{tabular}{|C{10mm}|R{2mm}R{6mm}R{8mm}R{10mm}R{12mm}R{16mm}R{18mm}|C{15mm}|}
\hline
$ n $ & 1\ic & 2\ic & 3\ic & 4\ic & 5\ic & 6\ic & 7 & OEIS \\
\hhline{|=|=======|=|}
$ L(\mathcal{B}_{n}) $ & 1\ic & 2\ic & 5\ic & 15\ic & 52\ic & 203\ic & 877 & A000110 \\
\hline
$ L(\mathcal{C}^{1}_{n}) $ & 1\ic & 4\ic & 23\ic & 173\ic & 1602\ic & 17575\ic & 222497 & A075729 \\
\hline
$ L(\mathcal{C}^{2}_{n}) $ & 1\ic & 6\ic & 53\ic & 619\ic & 8972\ic & 155067\ic & 3109269 & A109092  \\
\hline
$ L(\mathcal{C}^{3}_{n}) $ & 1\ic &  8\ic &  95\ic &  1497\ic &  29362\ic &  688439\ic &  18766393 & None  \\
\hline
$ L(\mathcal{C}^{4}_{n}) $ & 1\ic & 10\ic & 149\ic & 2951\ic & 72852\ic & 2152651\ic & 74031869 & None \\
\hline
\end{tabular}
\caption{The numbers of flats of the extended Catalan arrangements}
\end{table}

\begin{table}[H]
\centering
\begin{tabular}{|C{10mm}|R{2mm}R{6mm}R{8mm}R{10mm}R{12mm}R{16mm}R{18mm}|C{15mm}|}
\hline
$ n $ & 1\ic & 2\ic & 3\ic & 4\ic & 5\ic & 6\ic & 7 & OEIS \\
\hhline{|=|=======|=|}
$ L(\mathcal{S}^{1}_{n}) $ & 1\ic & 3\ic & 13\ic & 73\ic & 501\ic & 4051\ic & 37633 & A000262 \\
\hline
$ L(\mathcal{S}^{2}_{n}) $ & 1\ic & 5\ic & 37\ic & 361\ic & 4361\ic & 62701\ic & 1044205 & A025168  \\
\hline
$ L(\mathcal{S}^{3}_{n}) $ & 1\ic & 7\ic & 73\ic & 1009\ic & 17341\ic & 355951\ic & 8488117 & A321837  \\
\hline
$ L(\mathcal{S}^{4}_{n}) $ & 1\ic & 9\ic & 121\ic & 2161\ic & 48081\ic & 1279801\ic & 39631369 & A321847  \\
\hline
$ L(\mathcal{S}^{5}_{n}) $ & 1\ic & 11\ic & 181\ic & 3961\ic & 108101\ic & 3532651\ic & 134415961 & A321848 \\
\hline
\end{tabular}
\caption{The numbers of flats of the extended Shi arrangements}
\end{table}

\begin{table}[H]
\centering
\begin{tabular}{|C{10mm}|R{2mm}R{6mm}R{8mm}R{10mm}R{12mm}R{16mm}R{18mm}|C{15mm}|}
\hline
$ n $ & 1\ic & 2\ic & 3\ic & 4\ic & 5\ic & 6\ic & 7 & OEIS \\
\hhline{|=|=======|=|}
$ L_{1}(\mathcal{B}_{n}) $ & 1\ic & 1\ic & 1\ic & 1\ic & 1\ic & 1\ic & 1 & A000012 \\
\hline
$ L_{1}(\mathcal{C}^{1}_{n}) $ & 1\ic & 3\ic & 13\ic & 75\ic & 541\ic & 4683\ic & 47293 & A000670 \\
\hline
$ L_{1}(\mathcal{C}^{2}_{n}) $ & 1\ic & 5\ic & 37\ic & 365\ic & 4501\ic & 66605\ic & 1149877 & A050351  \\
\hline
$ L_{1}(\mathcal{C}^{3}_{n}) $ & 1\ic & 7\ic & 73\ic & 1015\ic & 17641\ic & 367927\ic & 8952553 & A050352  \\
\hline
$ L_{1}(\mathcal{C}^{4}_{n}) $ & 1\ic & 9\ic & 121\ic & 2169\ic & 48601\ic & 1306809\ic & 40994521 & A050353 \\
\hline
\end{tabular}
\caption{The numbers of $ 1 $-dimensional flats of the extended Catalan arrangements}
\end{table}

\begin{table}[H]
\centering
\begin{tabular}{|C{11mm}|R{2mm}R{6mm}R{8mm}R{10mm}R{12mm}R{16mm}R{18mm}|C{15mm}|}
\hline
$ n $ & 1\ic & 2\ic & 3\ic & 4\ic & 5\ic & 6\ic & 7 & OEIS \\
\hhline{|=|=======|=|}
$ L_{1}(\mathcal{S}^{1}_{n}) $ & 1\ic & 2\ic & 6\ic & 24\ic & 120\ic & 720\ic & 5040 & A000142 \\
\hline
$ L_{1}(\mathcal{S}^{2}_{n}) $ & 1\ic & 4\ic & 24\ic & 192\ic & 1920\ic & 23040\ic & 322560 & A002866  \\
\hline
$ L_{1}(\mathcal{S}^{3}_{n}) $ & 1\ic & 6\ic & 54\ic & 648\ic & 9720\ic & 174960\ic & 3674160 & A034001  \\
\hline
$ L_{1}(\mathcal{S}^{4}_{n}) $ & 1\ic & 8\ic & 96\ic & 1536\ic & 30720\ic & 737280\ic & 20643840 & A034177  \\
\hline
$ L_{1}(\mathcal{S}^{5}_{n}) $ & 1\ic & 10\ic & 150\ic & 3000\ic & 75000\ic & 2250000\ic & 78750000 & A034325 \\
\hline
\end{tabular}
\caption{The numbers of $ 1 $-dimensional flats of the extended Shi arrangements}
\end{table}

\begin{table}[H]
\centering
\begin{tabular}{|C{11mm}||R{11mm}R{11mm}R{9mm}R{8mm}R{4mm}|}
\hline
\multicolumn{1}{|l}{$ L_{k}(\mathcal{B}_{n}) $ } & \multicolumn{5}{r|}{A008277} \\
\hline
$ n\backslash k $ & 1\ic & 2\ic & 3\ic & 4\ic & 5 \\
\hhline{|=#=====|}
1 & 1\ic &  &  &  &  \\
2 & 1\ic  & 1\ic &  &  &  \\
3 & 1\ic & 3\ic & 1\ic &  &  \\
4 & 1\ic & 7\ic & 6\ic & 1\ic &  \\
5 & 1\ic & 15\ic & 25\ic & 10\ic & 1 \\
\hline
\end{tabular}
\begin{tabular}{|C{11mm}||R{11mm}R{11mm}R{9mm}R{8mm}R{4mm}|}
\hline
\multicolumn{1}{|l}{$ L_{k}(\mathcal{S}^{1}_{n}) $ } & \multicolumn{5}{r|}{A105278} \\
\hline
$ n\backslash k $ & 1\ic & 2\ic & 3\ic & 4\ic & 5 \\
\hhline{|=#=====|}
1 & 1\ic &  &  &  &  \\
2 & 2\ic  & 1\ic &  &  &  \\
3 & 6\ic & 6\ic & 1\ic &  &  \\
4 & 24\ic & 36\ic & 12\ic & 1\ic &  \\
5 & 120\ic & 240\ic & 120\ic & 20\ic & 1 \\
\hline
\end{tabular} \\ \vspace{1mm}
\begin{tabular}{|C{11mm}||R{11mm}R{11mm}R{9mm}R{8mm}R{4mm}|}
\hline
\multicolumn{1}{|l}{$ L_{k}(\mathcal{C}^{1}_{n}) $ } & \multicolumn{5}{r|}{A088729} \\
\hline
$ n\backslash k $ & 1\ic & 2\ic & 3\ic & 4\ic & 5 \\
\hhline{|=#=====|}
1 & 1\ic &  &  &  &  \\
2 & 3\ic  & 1\ic &  &  &  \\
3 & 13\ic & 9\ic & 1\ic &  &  \\
4 & 75\ic & 79\ic & 18\ic & 1\ic &  \\
5 & 541\ic & 765\ic & 265\ic & 30\ic & 1 \\
\hline
\end{tabular}
\begin{tabular}{|C{11mm}||R{11mm}R{11mm}R{9mm}R{8mm}R{4mm}|}
\hline
\multicolumn{1}{|l}{$ L_{k}(\mathcal{S}^{2}_{n}) $ } & \multicolumn{5}{r|}{A079621} \\
\hline
$ n\backslash k $ & 1\ic & 2\ic & 3\ic & 4\ic & 5 \\
\hhline{|=#=====|}
1 & 1\ic &  &  &  &  \\
2 & 4\ic  & 1\ic &  &  &  \\
3 & 24\ic & 12\ic & 1\ic &  &  \\
4 & 192\ic & 144\ic & 24\ic & 1\ic &  \\
5 & 1920\ic & 1920\ic & 480\ic & 40\ic & 1 \\
\hline
\end{tabular}\\ \vspace{1mm}
\begin{tabular}{|C{11mm}||R{11mm}R{11mm}R{9mm}R{8mm}R{4mm}|}
\hline
\multicolumn{1}{|l}{$ L_{k}(\mathcal{C}^{2}_{n}) $ } & \multicolumn{5}{r|}{A308440} \\
\hline
$ n\backslash k $ & 1\ic & 2\ic & 3\ic & 4\ic & 5 \\
\hhline{|=#=====|}
1 & 1\ic &  &  &  &  \\
2 & 5\ic  & 1\ic &  &  &  \\
3 & 37\ic & 15\ic & 1\ic &  &  \\
4 & 365\ic & 223\ic & 30\ic & 1\ic &  \\
5 & 4501\ic & 3675\ic & 745\ic & 50\ic & 1 \\
\hline
\end{tabular}
\begin{tabular}{|C{11mm}||R{11mm}R{11mm}R{9mm}R{8mm}R{4mm}|}
\hline
\multicolumn{1}{|l}{$ L_{k}(\mathcal{S}^{3}_{n}) $ } & \multicolumn{5}{r|}{A308281} \\
\hline
$ n\backslash k $ & 1\ic & 2\ic & 3\ic & 4\ic & 5 \\
\hhline{|=#=====|}
1 & 1\ic &  &  &  &  \\
2 & 6\ic  & 1\ic &  &  &  \\
3 & 54\ic & 18\ic & 1\ic &  &  \\
4 & 648\ic & 324\ic & 36\ic & 1\ic &  \\
5 & 9720\ic & 6480\ic & 1080\ic & 60\ic & 1 \\
\hline
\end{tabular}\\ \vspace{1mm}
\begin{tabular}{|C{11mm}||R{11mm}R{11mm}R{9mm}R{8mm}R{4mm}|}
\hline
\multicolumn{1}{|l}{$ L_{k}(\mathcal{C}^{3}_{n}) $ } & \multicolumn{5}{r|}{NONE} \\
\hline
$ n\backslash k $ & 1\ic & 2\ic & 3\ic & 4\ic & 5 \\
\hhline{|=#=====|}
1 & 1\ic &  &  &  &  \\
2 & 7\ic  & 1\ic &  &  &  \\
3 & 73\ic & 21\ic & 1\ic &  &  \\
4 & 1015\ic & 439\ic & 42\ic & 1\ic &  \\
5 & 17641\ic & 10185\ic & 1465\ic & 70\ic & 1 \\
\hline
\end{tabular}
\begin{tabular}{|C{11mm}||R{11mm}R{11mm}R{9mm}R{8mm}R{4mm}|}
\hline
\multicolumn{1}{|l}{$ L_{k}(\mathcal{S}^{4}_{n}) $ } & \multicolumn{5}{r|}{A048786} \\
\hline
$ n\backslash k $ & 1\ic & 2\ic & 3\ic & 4\ic & 5 \\
\hhline{|=#=====|}
1 & 1\ic &  &  &  &  \\
2 & 8\ic  & 1\ic &  &  &  \\
3 & 96\ic & 24\ic & 1\ic &  &  \\
4 & 1536\ic & 576\ic & 48\ic & 1\ic &  \\
5 & 30720\ic & 15360\ic & 1920\ic & 80\ic & 1 \\
\hline
\end{tabular}\\ \vspace{1mm}
\begin{tabular}{|C{11mm}||R{11mm}R{11mm}R{9mm}R{8mm}R{4mm}|}
\hline
\multicolumn{1}{|l}{$ L_{k}(\mathcal{C}^{4}_{n}) $ } & \multicolumn{5}{r|}{NONE} \\
\hline
$ n\backslash k $ & 1\ic & 2\ic & 3\ic & 4\ic & 5 \\
\hhline{|=#=====|}
1 & 1\ic &  &  &  &  \\
2 & 9\ic  & 1\ic &  &  &  \\
3 & 121\ic & 27\ic & 1\ic &  &  \\
4 & 2169\ic & 727\ic & 54\ic & 1\ic &  \\
5 & 48601\ic & 21735\ic & 2425\ic & 90\ic & 1 \\
\hline
\end{tabular}
\begin{tabular}{|C{11mm}||R{11mm}R{11mm}R{9mm}R{8mm}R{4mm}|}
\hline
\multicolumn{1}{|l}{$ L_{k}(\mathcal{S}^{5}_{n}) $ } & \multicolumn{5}{r|}{A308282} \\
\hline
$ n\backslash k $ & 1\ic & 2\ic & 3\ic & 4\ic & 5 \\
\hhline{|=#=====|}
1 & 1\ic &  &  &  &  \\
2 & 10\ic  & 1\ic &  &  &  \\
3 & 150\ic & 30\ic & 1\ic &  &  \\
4 & 3000\ic & 900\ic & 60\ic & 1\ic &  \\
5 & 75000\ic & 30000\ic & 3000\ic & 100\ic & 1 \\
\hline
\end{tabular}
\caption{Triangles of the number of $ k $-dimensional flats}\label{tab:triangles}
\end{table}

\newpage
\bibliographystyle{amsplain}
\bibliography{bibfile}

\end{document}